\documentclass[twoside]{amsart}
\usepackage{amssymb}
\usepackage{mathrsfs}
\usepackage{graphicx}
\usepackage[latin1]{inputenc}
\usepackage[english]{babel}

\newcommand{\F}{ \mathscr F }
\newcommand{\sL}{\mbox{${\mathscr L}$}}
\newcommand{\sA}{\mbox{${\mathscr A}$}}

\newcommand{\sK}{\mbox{${\mathscr K}$}}
\newcommand{\sG}{\mbox{${\mathscr G}$}}

\newcommand{\E}{\mbox{${\mathscr E}$}}

\renewcommand{\O}{{\mathcal O}}
\renewcommand{\H}{\text{\rm H}}

\renewcommand{\dim}{\mathrm{dim}\,}

\def\H{\mathscr H}

\input{xypic.tex} \xyoption{all}

\newtheorem{lema}{Lemma}[section]
\newtheorem{cor}[lema]{Corollary}
\newtheorem{teo}[lema]{Theorem}

\newtheorem{prop}[lema]{Proposition}
\theoremstyle{definition}

\newtheorem{remark}[lema]{Remark}
\newtheorem{defi}[lema]{Definition}

\newtheorem{exe}[lema]{Example}

\begin{document}

\title[Rational  Morita equivalence for holomorphic Poisson modules]
{ Rational  Morita equivalence for holomorphic Poisson modules}

\author[M. Corr\^ea]{M. Corr\^ea}
\thanks{ }
\dedicatory{}
\address{ICEX-UFMG \\ Departamento de Matem\'atica \\
Av. Ant\^onio Carlos, 6627 \\ 31270-901 Belo Horizonte-MG, Brazil}
\email{mauriciojr@ufmg.br}


% ----------------------------------------------------------------------

\begin{abstract}
We introduce a  weak concept of Morita equivalence,  in the birational context, for   Poisson modules on complex  normal Poisson projective   varieties.  We show that Poisson modules,
on   projective varieties with mild   singularities,  are  either rationally  Morita  
equivalent to a flat partial holomorphic  sheaf, or   a sheaf  with a meromorphic  flat connection or a co-Higgs  sheaf.  As an application, we  
study the geometry of rank two  meromorphic rank two  $\mathfrak{sl}_2$-Poisson modules 
which can be interpreted as a Poisson  analogous   to transversally projective structures for codimension one holomorphic foliations. Moreover, we describe the geometry of the symplectic foliation induced by the  Poisson connection on the  projectivization  of the Poisson module.
\end{abstract}

% ----------------------------------------------------------------------
%\begin{center}
\maketitle
%\end{center}
% ----------------------------------------------------------------------
%\tableofcontents

\section{Introduction}

K. Morita in  his celebrated work  \cite{Morita}  introduced  an equivalence in algebra 
 proving  that  two rings have equivalent categories of left
modules if and only if there exists an equivalence bimodule for the rings. 
Weinstein   \cite{W3}  and Xu \cite{X1}  have  introduced a  geometric Morita equivalence in the context of integrable Poisson real manifolds having as one of the motivations  the fact that   symplectic realizations of Poisson manifolds   is the analogous to  representations
 of associative algebras.   For non-integrable Poisson manifolds, an infinitesimal notion of Morita equivalence has been 
  introduced  by Crainic in  \cite{Crainic} and 
 by  Ginzburg  in \cite{G} in order to study the invariance, respectively,  of  Poisson cohomology and  Poisson  Grothendieck groups. 

In this work we introduce  a    weak concept of Morita equivalence  in the birational context. More precisely, 
 we say that two Poisson normal projective varieties  $(X,\sigma_1)$ and  $  (Y,\sigma_2) $   are  {\it rationally   Morita equivalent   }  
 if there exists  a   normal variety $ (S,\varrho)$, with a (possibly meromorphic ) Poisson bivector $\varrho$, and  two  arrows  
$$
\xymatrix@R=10pt{
 &(S,\varrho) \ar@{->}[dl]_{f} \ar@{->}[dr]^{h}& \\
(X,\sigma_1)& & (Y,\sigma_2),}
$$
such that  $f$ and $h$ are dominant     Poisson   morphisms. Our main aim
is to  show  that  holomorphic Poisson  modules,  on  projective  varieties  with  mild singularities,  are  either  rationally  Morita  equivalent  to  a  flat  partial  holomorphic  sheaf,  or a sheaf with a meromorphic flat connection or a co-Higgs sheaf.

A. Polishchuk in \cite{P} has studied the algebraic geometry of Poisson modules  motived by  Bondal's conjecture  about the non-triviality of  the  degeneration locus  of a  Poisson structure, see   also    \cite{Beauville2, GPym1,  Druel2}.  Poisson modules also appear in the context of generalized complex geometry introduced by N. Hitchin in \cite{H3} and developed by  M.  Gualtieri in \cite{Gualtieri2}.  
 Gualtieri's concept of a generalized holomorphic bundle \cite{Gualtieri1} in the Poisson case coincides with the notion of
Poisson modules and if the Poisson structure is the trivial one  we obtain  a co-Higgs bundle \cite{H}. See also   \cite{BGHR,Rayan,C,BHu, BS, Col} for more details  on the study of  co-Higgs sheaves and their moduli spaces.

A Poisson structure on  a  projective complex manifold $X$ induces  a natural Poisson structure on its   minimal model by pushing forward the Poisson bivector.
Thus, it is natural to consider Poisson varieties with mild singularities motived by the   development of the minimal model program.   
We say that a Poisson variety  $(X,\sigma)$  is {\it  klt }    if $X$ is a   Kawamata log terminal(klt)  variety   and the Poisson structure $\sigma$ is  either   generically symplectic   or  the associated  symplectic foliation  $\F_{\sigma}$ 
has canonical  singularities.   We say that  $(X,\sigma)$  is  {\it transcendental}, in the spirit of \cite{ AD2, LPT}, if there is no positive-dimensional algebraic
subvariety through a general point of $X$ that is tangent to the  symplectic  foliation $\F_{\sigma}$.

\begin{teo}\label{teo-mori}
Let  $(E, \nabla) $ be a  locally free Poisson module   on a    klt   Poisson  projective  variety $(X,\sigma)$.  
Then  at least one of the following holds up to  rational   Morita equivalence. 
\begin{itemize}
\item[(a)]   $(E, \nabla) $ corresponds to  a   flat holomorphic sheaf   on a transcendental Poisson  variety;

\item[(b)]    $(E, \nabla) $ corresponds to a  meromorphic    flat   connection on a generically symplectic variety. 

\item[(c)]      $(E, \nabla) $ corresponds  to a co-Higgs  sheaf on a   variety with trivial Poisson structure . 

\item[(d)]      $(E, \nabla) $ corresponds  to a meromorphic  co-Higgs  sheaf  $(E_0, \psi) $  
  on a transcendental Poisson variety $(Y,\sigma_0)$, there exist
a rational map $\zeta:Y  \dashrightarrow B$,  over a variety $B$ with $\dim(B)= \dim(\F_{\sigma_0})$,  and the co-Higgs field  
$ \phi$ is tangent to $T_{Y|B}$  and  satisfies $D_0(\phi)=0$, where $D_0$ is  a  meromorphic extension of a  Poisson connection on $T_{Y|B}\otimes End(E_0)$.
\end{itemize} 
\end{teo}

As an  application of Theorem \ref{teo-mori}, we  provide a structure theorem for  rank two holomorphic  $\mathfrak{sl}_2$-Poisson modules.

\begin{cor} \label{aplic}
Let  $(E,\nabla)$ be a  rank two  holomorphic  $\mathfrak{sl}_2$-Poisson   module  on a  klt Poisson  projective  variety $(X,\sigma)$.  
 Then there exist   projective
varieties $Y$ and $Z$ with klt singularities, a 
quasi-\'etale  Poisson cover $f: W\times Y \to  X$ and  at least one of the following holds.  
\begin{itemize}
\item[(a)]    $(\pi_2)_*f^*(E,\nabla)$  is a  $\mathfrak{sl}_2$ partial   holomorphic  sheaf on $Y$, where $\pi_2$ denotes the projection on $Y$. 

\item[(b)] If $W$ and $Y$ are  generically symplectic, then $(\pi_2)_*f^*(E,\nabla)$ is a rank two  locally free sheaf 
   with a meromorphic    flat   connection with poles on the degeneracy Poisson divisor of $Y$. 

 \item[(c)]  
 If $W$ is    symplectic, then   after a birational trivialization  of $ f^*(E,\nabla)$ the   Poisson connection on the trivial bundle is defined  as 
$$
\tilde{ \nabla}=\delta_{W}+\begin{pmatrix}f_1 &f_2
 \\f_0&-f_1\end{pmatrix}\otimes v,
$$
for some rational  vector field $v$ tangent to $(Y,0)$,  rational functions $f_0,f_1,f_2 \in K(Y)$, and 
$ \delta_{W}$ denotes the  Poisson differential on $W$.
\item[(d)]    There exists a rational  map  $\zeta: Y  \dashrightarrow B$,  
over a variety $B$ with $\dim(B)= \dim(\F_{\alpha})$,   such that $(\pi_2)_*f^*(E,\nabla)$ 
corresponds to a meromorphic    $\mathfrak{sl}_2$-Poisson 
 module  $(E_0, \tilde{ \nabla}) $, such that after a birational trivialization  the  Poisson connection  on the trivial bundle is defined  as 
$$
\tilde{ \nabla}=\delta+\begin{pmatrix}f_1 &f_2
 \\f_0&-f_1\end{pmatrix}\otimes v,
$$
for some rational Poisson vector field $v$ and    rational functions $f_0,f_1,f_2$ on $X$ such that $\{f_i,f_j\}= 0$, for all $i,j$. 
\end{itemize}
\end{cor}

In the  section \ref{section-sl2} we  point out   that the   study  of  rank two  $\mathfrak{sl}_2$-Poisson modules is equivalent to  the understanding of the following objects:
\begin{itemize}
\item[i) ] a triple of rational vector fields $(v_0,v_1,v_2)$  on $X$ such that 
\begin{equation}\label{MCintro}
 \begin{matrix}
\delta(v_0)=\hfill v_0\wedge v_1\\
\delta(v_1)=2v_0\wedge v_2\\
\delta(v_2)=\hfill v_1\wedge v_2,
\end{matrix} 
\end{equation}

\item[ii) ]  the symplectic foliation  $\F_{ \nabla}$  corresponding to  {\it Polishchuk's Poisson structure}  induced by $\nabla$  on  $\pi:\mathbb{P}(E,\nabla)\to (X,\sigma)$. 

\end{itemize}

We  can say  that the study of  the triples $(v_0,v_1,v_2)$  
satisfying (\ref{MCintro}) is the Poisson    analogous   to   transversely projective holomorphic foliations theory due to B. Sc\'ardua \cite{S}.
We refer to the works  \cite{ Cousin, CP,  LPT2},  where  the authors have studied transversely projective   foliations via meromorphic connections on rank two vector bundles \cite{D}. 

From Corollary \ref{aplic} we observe  that the  geometric study of  the symplectic foliation  $\F_{ \nabla}$  reduces, up to a 
quasi-\'etale  Poisson cover,   to the   foliation $\F_{ \nabla_0}$  on    $  \mathbb{P}(E_0,  \nabla_0) \to (Y, \sigma_0)$.  Our next result describes the geometry of such foliation.

\begin{cor} Under the same assumptions as Corollary \ref{aplic}. Let $\F_{ \nabla_0}$ be the symplectic foliation induced on    $  \mathbb{P}(E_0) \to (Y, \sigma_0)$.  
Then   at least one of the following holds.

\begin{itemize}
\item[(a)] 

$\F_{ \nabla_0}$   is a  dimension 2 foliation   which is a pull-back of a foliation  by curves on $(Y,0)$.  

\item[(b)]  
 $ \F_{ \nabla_0}$  is a Riccati  foliation of codimension  one on     $ \mathbb{P}(E_0)$, if  $ (Y, \sigma_0)$ is generically symplectic.     

\item[(c)]  
$\F_{ \nabla_0}$  is a  Riccati  foliation of codimension  one on   $ \mathbb{P}(E_0)$ which  is given by  a morphism  $\sA \to  d_{\text{\rm refl}} \pi(  \pi^{*}(T\F_{ \sigma_0}^*))\subset \Omega_{\mathbb{P}(E_0)}^{[1]}$, where $\sA$ is  a line bundle and  $d_{\text{\rm refl}} \pi: \pi^{*}\Omega_Y^{[1]} \to \Omega_{ \mathbb{P}(E_0) }^{[1]} $  is the  pull-back morphism of reflexive forms.

\item[(d)]   There exist a rational Poisson vector field $v$ generically transversal to $\F_{\sigma_0}$  such that  $\F_{ \nabla_0}$     has   dimension  $2k+2$ and it is the  pull-back of the foliation induced by $v$ and $\F_{\sigma_0}$.  In particular, if $\dim(Y)=2k+1$, then  $ \mathbb{P}(E_0)$ is generically symplectic and there exist a rational Poisson map $\zeta:  Y  \dashrightarrow  B$ generically transversal to  $\F_{\sigma_0}$, where $B$ is a generically symplectic variety with  $\dim(B)=2k$ and  the induced map  $ \mathbb{P}(E_0) \dashrightarrow  B$ is  Poisson. 
  
\end{itemize}

\end{cor}

\subsection*{Acknowledgments}
 I would like to express my deep gratitude to Nigel Hitchin for  
 his generosity and  patience. I am  honoured for the opportunity  of  having   
many illuminating discussions with him during my visit to the  University of Oxford. 
I would also like to thank Marcos Jardim, Alex Massarenti, Henrique Bursztyn,  Steven Rayan and  Brent Pym for useful discussions. I  would like to thank   the referee for precious comments  which improved the presentation of the paper greatly.
I am   grateful to   Mathematical Institute of the University of Oxford and L'Universit\`a degli Studi di Ferrara
for their hospitality. 
 This work was   partially funded by CNPQ grant numbers 202374/2018-1, 302075/2015-1, 400821/2016-8 ,   and  CAPES  grant number 	
2888/2013.

\section{Holomorphic foliations}
Throughout this  paper a  variety $X$ is  a scheme of finite type over $\mathbb{C}$ and  regular in codimension one, with its smooth locus denoted by  $X_{reg}$. 
As usual,  consider  $\O_X$ the sheaf of germs of  holomorphic functions on $X$. We  denote by  $TX=\mathcal{H}om(\Omega_X^1,\O_X)$ the tangent sheaf of $X$. 
Given   $p\in \mathbb{N}$, we   denote by $\Omega_X^{[p]}$
  the sheaf $(\Omega_X^p)^{**}$.  

Let $X$ be a normal variety and suppose that   $ K_X$  is  $\mathbb{Q}$-Cartier, i.e., some
nonzero multiple of it is a Cartier divisor. 
Consider a resolution of singularities $f:Z\to X$. There are uniquely defined  rational  numbers $a(E_i, X)$'s 
such that
$$K_{Z}=f^{*}K_{X} +\sum_{i} a(E_i, X) E_i,$$ 
where $E_i$ runs through all exceptional prime divisors for $f$. The   $a(E_i, X)$'s do not depend on the   resolution $f$, but only on the valuations
associated to the $E_i$. For more details we refer  to  \cite[Section 2.3]{KM}.

We say that $X$ is \emph{  Kawamata log terminal (klt) }  
 if  all $a(E_i, X )> -1$ for every  $f$-exceptional prime divisor $E_i$, 
 for  some resolution $f:Z\to X$.

\begin{defi}
A \emph{foliation}  $\F$ on  a normal variety   $X$ is a coherent subsheaf $T\F\subset TX$ such that
 $\F$ is closed under the Lie bracket, and
the dimension  of $\F$ is the generic rank of $T\F$. The \textit{canonical bundle }  is defined by $K_{\F} = \det(T\F)^{**}$. 
\end{defi}
We will denote by $a: T\F \to TX$ an injective morphism (anchor map)  defining the foliation  $\F$. 

The singular set of $\F$ is defined by $Sing(\F)=Sing(N\F)$, where $N\F=T_X/T\F$ is the normal sheaf of the foliation. 
Hereafter we  will suppose that $cod(Sing(\F))\geq 2$.
We have a exact sequence of sheaves 
$$
0\to T\F \to TX \to  N\F\to 0. 
$$
\begin{defi}\cite{Mc1}
Let $\F$ be a holomorphic foliation on a projective variety  $X$ and 
 $f:Y\to X$   a projective birational morphism. 
We  say that $\F$ has  \emph{canonical}  singularities if 
the divisor 
$K_{f^{-1}\F}- f^{*}K_{\F} $ 
is effective.
\end{defi}

\section{Poisson modules}

A Poisson structure
on  a variety $X$ is a  $\mathbb{C}$-linear Lie bracket $$\{ \cdot , \cdot \}: \O_X  \times  \O_X  \to \O_X$$ 
which  satisfies the Leibniz rule  $\{ f , gh \} = h \{f,g \} + f\{ g,h\}$  
and   Jacobi identities
$$\{f, \{g, h\}\} + \{g, \{h, f\}\} + \{h, \{f, g\}\} = 0$$
for all germs of holomorphic functions $f,g,h$.   The bracket
corresponds  to a bivector field   $\sigma\in H^0(X, \wedge^2TX)$   given by
$\sigma(df \wedge  dg)= \{f, g\}$,  for all germs of holomorphic functions $f,g$. 

We will denote a Poisson structure on $X$ as the pair $(X, \sigma)$, where $\sigma\in H^0(X, \wedge^2TX)$ is the corresponding  Poisson bivector field.  
The bivector induces a morphism
$$
\sigma^{\#}: \Omega_X^1\to TX
$$
which is called  the {\it anchor map}   and it is  defined by $\sigma^{\#}(\theta)=\sigma(\theta,\cdot)$, where $\theta$ is a germ of holomorphic $1$-form.
\begin{defi}
The {\it symplectic foliation}  associated to $\sigma$ is the foliation given by $\F_{\sigma}:= \mathrm{Ker}(\sigma^{\#})$, whose dimension is the rank the anchor map 
$\sigma^{\#}: \Omega_X^1\to TX$. 
A Poisson variety   $(X,\sigma)$  is called \emph{generically  symplectic}      if the anchor map   $\sigma^{\#}: \Omega_X^1\to TX$ is generically an isomorphism. 
Then, the degeneracy loci of  $\sigma^{\#}$ is an effective anti-canonical  divisor $D(\sigma)\in |-K_X|$.
\end{defi}

A meromorphic Poisson bivector  is a meromorphic section $\sigma$ of $ \wedge^2TX$ such that $[\sigma,\sigma]=0$, where $[\ ,\ ]$  denotes the Schouten bracket. 
Observe that in this case  we have a Poisson structure outside of  the poles divisor of $\sigma$.

A {\it rational vector field}   is a section  $v\in H^0(X, TX\otimes \sL)$ for some invertible sheaf $\sL$. We say that 
 $v$ is a  \emph{Poisson}  rational vector field with respect to $\sigma$ if it is an infinitesimal
symmetry of  $\sigma$, ie.,  $L_v(\sigma)=0$, where $L_v(\cdot)$ denotes the Lie derivative.

We denote by $\delta $ the corresponding Poisson differential.  For instance, $\delta (f)=- \sigma^{\#}(df)$ and  $\delta (v)= L_v(\sigma)$, for all germs of holomorphic function $f$ and germs of vector field $v$.

\begin{defi}
 We say that a Poisson variety  $(X,\sigma)$  is {\it  klt }    if $X$ is a  klt  variety   and the Poisson structure $\sigma$ is  either   generically symplectic   or  the associated  symplectic foliation  $\F_{\sigma}$ 
has canonical  singularities.
\end{defi}

\begin{defi} \cite{Atiyah}
Let $ X $ be a  projective variety.  A {\it holomorphic  connection}  on a 
 sheaf of $\O_X$-modules $E$ is  a 
$\mathbb{C}$-linear morphism of sheaves 
  $\nabla:E\to  \Omega_X^1\otimes E  $ satisfying the Leibniz rule
  $$
  \nabla(fs)= d (f)\otimes s +f \nabla(s),
  $$
 where  $f $ is  germ of holomorphic function on $X$ and $s$ is a germ of holomorphic section   of $E$. For a   holomorphic connection $\nabla$, one defines as usual its extension(curvature) $\nabla^2:E\to  \Omega_X^2\otimes E  $. We say that  $\nabla$ is flat if $\nabla^2=0$. 
We say that $E $ is 
 a {\it flat holomorphic  sheaf}  if it admits a   flat holomorphic  connection. 
   \end{defi}

\begin{defi}
Let $(X,\sigma)$ be a  Poisson projective variety.  A {\it Poisson connection}  on a 
 sheaf of $\O_X$-modules $E$ is  a 
$\mathbb{C}$-linear morphism of sheaves 
  $\nabla:E\to  T_X\otimes E  $ satisfying the Leibniz rule
  $$
  \nabla(fs)= \delta (f)\otimes s +f \nabla(s),
  $$
 where  $f $ is a germ of holomorphic function on $X$ and $s$ is a germ of holomorphic section of $E$. We say that $E$ is a \emph{Poisson module}  if  it admits a Poisson flat connection, i.e, if  its curvature  $\nabla^2:E\to  \Omega_X^2\otimes E  $ vanishes.
  Equivalently, a Poisson connection defines a  $\mathbb{C}$-linear bracket   $\{ \ ,\ \} : \O_{X} \times E  \to   E$ by
  $$
  \{ f,  s\}:=   \nabla(s)(df), 
  $$
  where  $f $ is a germ of holomorphic function on $X$ and $s$ is a germ of holomorphic section   of $E$. \end{defi}

\begin{defi}
Let $a:\sG \to TX$ be a holomorphic foliation.   A  \emph{ holomorphic partial connection}  on a 
 sheaf of $\O_X$-module $E$ is  a 
$\mathbb{C}$-linear morphism of sheaves 
  $\nabla:E\to   \sG^* \otimes E  $ satisfying the Leibniz rule
  $$
  \nabla(fs)= a^*(df)\otimes s +f \nabla(s),
  $$
 where  $f $ is a germ of holomorphic function on $X$,  $s$ is a germ of holomorphic section   of $E$   and  $a^*:\Omega_X^{[1]}\to  \sG^* $ denotes the dual map. If $E$ admits a flat holomorphic partial connection with respect a foliation $\sG$ we say that 
 $E$ is a  flat holomorphic  sheaf along $\sG$.   
\end{defi}

\begin{defi}\cite{Rayan2, H}
A \emph{co-Higgs sheaf}  on a  variety  $X$ is a sheaf $E$ together with a section 
$\phi \in H^0(X, TX\otimes End(E))$ (called
a co-Higgs fields) for which $\phi\wedge  \phi=0$. If $E$ is a locally free sheaf we say that it is a \emph{co-Higgs bundle}.

\end{defi}

\begin{defi}\cite[ section 2]{BGHR}
Let $G$ be a connected, reductive, affine algebraic group defined over $\mathbb{C}$.
A holomorphic principal $G$-bundle $E_G$ is a 
\emph{$G$-co-Higgs} bundle if there exist  a holomorphic section  $\phi$ of $ \mbox{ad}(E_G)\otimes TX$, where   $\mbox{ad}(E_G)$ denotes the adjoint vector bundle associated to $E_G$. If $\phi$ is a meromorphic section of $ \mbox{ad}(E_G)\otimes TX$ we say that $E_G$ is a meromorphic  $G$-co-Higgs bundle.  
\end{defi}
In this work we are interested in the case where $G=\mbox{SL}(2,\mathbb{C})$,  see section \ref{section-sl2}.  

\begin{exe}\label{exe-sym}
Let  $(X,\sigma)$ be a   generically  symplectic  variety of dimension $2n$  with   degeneracy divisor $D(\sigma)=D$.  It follows from \cite[Proposition 4.4.1]{Pym} that $\sigma$  
induces a skewsymmetric morphism 
$$
\sigma^{\#}:\Omega^1_X(\log D) \to T_X(-\log D)
$$ 
which is an isomorphism if and only if $D$ is reduced  \cite[Proposition 4.4.2]{Pym}. 
Therefore, if $D$ is reduced then  the isomorphism $\sigma^{\#}$ gives a   one-to-one correspondence  between Poisson flat connections  and logarithmic   flat   connections . 
If  $D$ is not  reduced, then a Poisson connection corresponds to a meromorphic flat connection with poles along $D$. 
 
\end{exe}

\begin{exe}
Let  $(X,\sigma)$ be a   Poisson   variety which is not generically  symplectic.   If  $E $ admites   a holomorphic flat connection    $\nabla:E \to  \Omega_X^1\otimes E $. Then
 $$\sigma^{\#} \circ \nabla: E \to  \Omega_X^1 \otimes E \to  T\F_{\sigma}\otimes E $$
 is a Poisson flat connection on $E$ tangent to the symplectic foliation $  \F_{\sigma}$. 
\end{exe}

\begin{exe}
Let  $(X,\sigma)$ be a   Poisson  projective variety whose the leaves of the  symplectic foliation $\F_{\sigma}$  are  the fibers of a rational map $\rho:X  \dashrightarrow  Y$. 
If  $\nabla:E \to \Omega_Y^1 \otimes E  $ is a holomorphic  flat connection, then $\rho^*E$  is a Poisson module on $X$, with
Poisson connection given by
$$
 \omega \circ \rho^*\nabla: \rho^*E \to  \Omega^1_{Y|X} \otimes  \rho^*E \to   T_{Y|X} \otimes  \rho^*E,
$$
where $ \omega:  \Omega^1_{Y|X} \to    T_{Y|X}  $ denotes the induced isomorphism given by $(\sigma^{\#}) _{|T\F_{\sigma}}$. 
\end{exe}

\begin{exe}
Let  $(X,0) $  be a     projective  variety with the trivial Poisson structure. 
Let $(E,\phi)$ be a  co-Higgs  bundle  on $X$. 
Then $ E$ is a Poisson module on $X$
whose Poisson connection is given by  $\nabla(fs):= \phi(df) s$, where  $f $ is a germ of holomorphic function on $X$,  $s$ is a germ of holomorphic section   of $E$.   

\end{exe}

\begin{exe}
Let  $(X,\sigma)$ be a   Poisson  projective manifold  and  $\rho:X   \dashrightarrow  Y$ a  rational map with connected fibers  such that $\rho_*\sigma=0$, then 
$\rho_*E$ is a co-Higgs sheaf on $Y$. In general, $\rho_*E$ is a Poisson module   on $(Y,\rho_*\sigma)$, since  $f_*\O_X  \simeq \O_Y$. 
See in the next section Proposition  \ref{Prop-Polishchuk}   due to Polishchuk. 
\end{exe}

\section{Rational Morita equivalence}

Le us recall the    notion of Morita equivalence of real  Poisson manifolds  which was developed
by Weinstein    \cite{W3}  and Xu \cite{X1} and  it  works  verbatim  in the  complex context:
we say that two Poisson  manifolds   $(X,\sigma_1)$ and  $  (Y,\sigma_2) $   are   Morita equivalence  
 if there exists  a  symplectic manifold  $ (S,\varrho)$  and with two arrows 
$$
\xymatrix@R=10pt{
 &(S,\varrho) \ar@{->}[dl]_{f} \ar@{->}[dr]^{h}& \\
(X,\sigma_1)& & (Y,-\sigma_2),}
$$
such that    are  Poisson  submersions. 
This  equivalence  has the following important properties :
\begin{itemize}
\item    there is  a bijection between
the leaves of the sympectic foliations of $(X_1,\sigma_1)$ and $(X_2,\sigma_2)$. 
\\ 
\item 
the space  of  Casimir functions  of  $(X_1,\sigma_1)$ and $(X_2,\sigma_2)$   are  isomorphic.  
\end{itemize}

\begin{defi} 
We say that two Poisson normal projective varieties   $(X,\sigma_1)$ and  $  (Y,\sigma_2) $   are  {\it rationally   Morita equivalence} 
 if there exists  a  projective    normal variety $ (S,\varrho)$, with a (possibly meromorphic ) Poisson bivector $\varrho$, and  two  diagrams 
$$
\xymatrix@R=10pt{
 &(S,\varrho) \ar@{->}[dl]_{f} \ar@{->}[dr]^{h}& \\
(X,\sigma_1)& & (Y,\sigma_2),}
$$
such that  $f$ and $h$ are dominant     Poisson   morphisms, i.e,  $ f_* \varrho= \sigma_1$ and $h_*  \varrho= \sigma_1$.
We say that two  Poisson modules  $E_1 \to (X,\sigma_1)$ and  $E_2  \to (Y,\sigma_2) $   are  {\it  rationally    Morita equivalence}  
if there is an equivalence $(X,\sigma_1)   \stackrel{ f }{ \longleftarrow}  (S,\varrho_{12})    \stackrel{ h }{\longrightarrow}  (Y,\sigma_2) $   such that  
  $ h_*  f^* E_1  $  and     $  E_2$ are  isomorphic  as Poisson modules. 
\end{defi}

We say that
 $$(X_1,\sigma_1)   \stackrel{ f_1 }{ \longleftarrow}  (S,\varrho_{12})    \stackrel{ h_1 }{\longrightarrow}  (X_2,\sigma_2) $$
 is isomorphic to 
 $$(X_3,\sigma_3)   \stackrel{ f_2 }{ \longleftarrow} (Q,\varrho_{34}) \stackrel{ h_2 }{ \longleftarrow}  (X_4,\sigma_4) $$
  if there is a birational map  $\zeta: (S,\varrho_{12})  \longrightarrow  (Q,\varrho_{34})$ such that $\zeta_* \varrho_{12}=\varrho_{34}$.
Therefore, for each  $(X_1,\sigma_1)   \stackrel{ f }{ \longleftarrow}  (S,\varrho_{12})    \stackrel{ h }{\longrightarrow}  (X_2,\sigma_2) $ we can  aways  take the resolution of singularities  $\zeta:\tilde{S}\to S$ and     the lifting  of  the meromorphic Poisson bivector field $\tilde{\varrho_{12}}$ will give us  that  $\zeta_* {\varrho_{12}}=\varrho_{12}$.  Thus, we obtain an equivalence

\centerline{
\xymatrix{
 &  (\tilde{S},\tilde{\varrho_{12}}) \ar[d]_{\zeta}  \ar@/^/@{->}[rdd]   \ar@/_/@{->}[ldd] & \\
 & (S,\varrho_{12}) \ar@{->}[dl]_{f} \ar@{->}[dr]^{g}& \\
(X,\sigma_1)& & (Y,\sigma_2),}
}

\noindent with $\tilde{S}$ smooth. 

  Since $f^{-1}\F_{\sigma_1} = h^{-1}\F_{\sigma_2}$ , we   obtain  the following.

\begin{prop}
The following holds:
\begin{itemize}
\item[1)]   there is  a bijection between
the leaves of the symplectic foliations of $(X_1,\sigma_1)$ and $(X_2,\sigma_2)$, 

\item[2)]  
the spaces  of  Casimir rational functions  of  $(X_1,\sigma_1)$ and $(X_2,\sigma_2)$   are  isomorphic,  
 
\item[3)] Morita equivalence implies the  rational Morita equivalence.  
\end{itemize}
\end{prop}

Let us list some known results:

\begin{prop} \label{Prop-Polishchuk}(Polishchuk \cite{P}) Let $X$ be a Poisson variety.  Let $f: X \to Y$ be a morphism such that $f_*\O_X \simeq \O_Y$. Then the Poisson structure on
$X$ induces canonically a Poisson structure on $Y$ such that $f$ is a Poisson morphism. Furthermore, if $\E$ is a
Poisson module on $X$, then  $f_*\E$ is a Poisson module on $Y$. 
\end{prop}

\begin{teo}\label{Kaledin} (Kaledin \cite{Kaledin})
Let $X$ be a Poisson variety. 
The reduction,  any completion and the normalization of $X$ are again Poisson varieties. 
\end{teo}
 Therefore, we can assume that the Poisson variety $X$ is reduced and normal. Moreover, if there exist a morphism  $f: X \to Y$  with connected fibers, then by   Proposition  \ref{Prop-Polishchuk}   we have that $X$ is rationally Morita equivalent  to $Y$,  
since $f_*\O_X  \simeq \O_Y$  \cite[Chapter 9]{FGLKNV}.
This also says us that a   Poisson structure on  a  projective complex manifold $X$ induces  a natural Poisson structure on its   minimal model by pushing forward the Poisson bivector.

\begin{exe}
Let $S$ be a smooth  Poisson surface with a Poisson structure $\sigma\in H^0(S,\O_S(-K_S))$. Let 
$S^{[r] }= \hbox{\rm Hilb}^{r}(S) $ the Hilbert scheme parametrizing 0-dimensional subschemes of $S$ of length $r$. 
Bottacin in  \cite{Bottacin}  extended   Beauville's construction of a symplectic
structure on the Hilbert scheme $S^{[r] }$, let us say $\sigma^{[r]}$,  of a symplectic  surface $(S,\sigma)$ \cite{Beauville}.
Consider $ \hbox{\rm Bl}_{\Delta}(S^2)$,    the blowup along the diagonal $\Delta\subset S^2=S\times S$. 
Ran in \cite{Ran} showed that  there is  a diagram
$$
\xymatrix@R=10pt{
 & \hbox{\rm Bl}_{\Delta}(S^2) \ar@{->}[dl]_{p} \ar@{->}[dr]^{q}& \\
S^{[2] }& &  S^{2 },}
$$
such that    $q^*(\sigma\times \sigma)$ is a meromorphic Poisson bivector on $\hbox{\rm Bl}_{\Delta}(S^2)$  with simple pole on $\mathbb{P}(\Omega^1_S)$ and $p_*q^*(\sigma\times \sigma)$  is  the Poisson structure $\sigma^{[2]}$ on $S^{[2] }$ . Therefore, $S^{[2] } $ is rationally Morita equivalent to $ S^{2 }$. 
In the general,  we can conclude by Ran's induction construction that  $S^{[r] } $ is rationally Morita equivalent to $ S^{[r-1] }\times S$, for all $r\geq  2$, see \cite[Section 1.6]{Ran}.
Since  $ S^{[r-1] }\times S$ is clearly Morita equivalent  to $S$, we conclude that  $S^{[r] } $ is Morita equivalent  to $S$, for all $r\geq  2$. 
\end{exe}

\begin{defi}
Let $(X,\sigma)$ be a Poisson variety  which is not generically symplectic. 
We say that  $(X,\sigma)$  is  {\it transcendental}, in the spirit of \cite{ AD2, LPT}, if there is no positive-dimensional algebraic
subvariety through a general point of $X$ that is tangent to the symplectic  foliation $\F_{\sigma}$. 
\end{defi}

\begin{defi}
A morphism  $f:Z\to X$ between normal varieties is called a {\it quasi-\'etale morphism}  if $f$ is finite and \'etale in
codimension one.
 \end{defi}

Now, we will prove our main result.
\begin{teo}\label{teo-Morita}
Let  $(E, \nabla) $ be a  locally free Poisson module   on a    klt   Poisson  projective  variety $(X,\sigma)$.  
Then  at least one of the following holds up to  rational   Morita equivalence.
\begin{itemize}
\item[(a)]   $(E, \nabla) $ corresponds to  a   flat holomorphic sheaf   on a transcendental Poisson  variety;

\item[(b)]    $(E, \nabla) $ corresponds to a  meromorphic    flat   connection on a generically symplectic variety. 

\item[(c)]      $(E, \nabla) $ corresponds  to a co-Higgs  sheaf on a   variety with trivial Poisson structure . 

\item[(d)]      $(E, \nabla) $ corresponds  to a meromorphic  co-Higgs  sheaf  $(E_0, \psi) $  
  on a transcendental Poisson variety $(Y,\sigma_0)$, there exist
a rational map $\zeta:Y  \dashrightarrow B$,  over a variety $B$ with $\dim(B)= \dim(\F_{\sigma_0})$,  and the co-Higgs field  
$ \phi$ is tangent to $T_{Y|B}$  and  satisfies $D_0(\phi)=0$, where $D_0$ is  a  meromorphic extension of a  Poisson connection on $T_{Y|B}\otimes End(E_0)$.
\end{itemize}
\end{teo}

\begin{proof}
Let   $(X,\sigma)$  be a  generically symplectic variety.     Consider the  degeneracy Poisson divisor
$D=\{\sigma^{n}=0\} \in |-K_X|$. Then, it follows from Example \ref{exe-sym}  that  $ \nabla$ induces  a  meromorphic    connection with poles along $D$ .

Suppose that $(X,\sigma)$  is  not    generically   symplectic and   consider  the  associated   symplectic foliation $\F_{\sigma}$. 
Since $K\F_{\sigma}\simeq\O_X$ and $\F_{\sigma}$ has canonical singularities, 
it follows from \cite[Proposition 8.14]{Druel18}   that  there exist two arrows 
$$
\xymatrix@R=10pt{
 &Z=W\times Y   \ar@{->}[dl]_{f} \ar@{->}[dr]^{\pi_2}& \\
X & & Y }
$$
such that  $ f: Z \to   X $ is a quasi-\'etale cover,  $ \pi_2: Z   \to  Y $  is the natural projection, $Y$ and $Z$ are 
normal klt projective varieties, 
and  it there is a  transcendental foliation $\sK$ on $Y$ such that  $\pi_2^{-1}\sK= f^{-1} \F_{\sigma} $.  

Since  $f: Z \to   X $ is a quasi-\'etale cover there is a Poisson bivector $\tilde{\sigma} \in H^0(Z,\wedge^2TZ  ) $  such that $f_*\tilde{\sigma}=\sigma.$ That is, $f$ is a Poisson morphism.  Indeed,   the map $ f: f^{-1}(X_{reg}) \to   X_{reg}$ is  a 
 map between complex manifolds which is a local biholomorphism, so there is a well-defined
pull back  Poisson bivector field  $(f^{-1})^*(\sigma|_{X_{reg}})$ which extends to a section 
 $\tilde{\sigma} \in H^0(Z,\wedge^2TZ  ) $. 
Similarly,   one can also see that $f^*E$ is a Poisson module, with respect to $\tilde{\sigma}$, by lifting the local matrices of vector fields  which  represent  the Poisson connection $\nabla$.  
Let  $\tilde{\nabla}$ denote   such induced  Poisson connection on $f^*E$. 
Thus, we have a Morita equivalence 
 $$(X,\sigma)   \longleftarrow (Z, \tilde{\sigma})   \longrightarrow  (Y,\sigma_2), $$
 where  
$\sigma_2=:(\pi_2)_*\tilde{\sigma}$.
In particular, we have that   $ E_1$  is rationally Morita equivalent   to  $(\pi_2)_*(f^*E):=E_0$. 
Denote by $\nabla_0$ the Poisson connection induced on $E_0$.

 Let us make the following simple but important observation:
 
  Let     $(E,\nabla) \to (W, \sigma)$  be a   Poisson module with a  flat Poisson connection $\nabla$ on a Poisson variety $(W,  \sigma)$ such that   $Sing(\F_{\sigma})\cup Sing(W)$ has codimension $\geq 2$. 
Then, we have that either:
  \begin{itemize}
\item[(i)]   $E$ corresponds to a  flat holomorphic  sheaf along the symplectic foliation $\F_{\sigma}$, or
\\ 
\item[(ii)]  $\nabla$  induces 
a non-trivial  section 
 $\phi\in H^0(W, N\F_{\sigma}\otimes End(E) )$
 such that $\phi\wedge \phi =0$. 
\end{itemize}
In fact, 
 if     $(E,\nabla)$  is  such that  $\nabla: E  \to T \F_{\sigma}\otimes E $, then $E$ corresponds to a  flat holomorphic  sheaf along the symplectic foliation.  
Indeed,   since $cod( Sing(\F_{\sigma}))\geq 2$,  then
the map 

\centerline{
\xymatrix{
T\F_{\sigma}^*  \ar[rr]_{i} \ar@/^2pc/[rrrr]^{\vartheta} &&  \Omega_W^1 \ar[rr]_{\sigma_2^{\#}} && T \F_{\sigma} \\
}
}
 \noindent  is an isomorphism, where $i:  T\F_{\sigma} ^* \to \Omega_W^1$ is the inclusion and $ \sigma_2^{\#}:  \Omega_W^1  \to  T\F_{\sigma} $ denotes the anchor map. 
Therefore, the Poisson connection 
 
 \centerline{
\xymatrix{
\vartheta^{-1}\circ \nabla_0: E \ar[r] &   \F_{\sigma}^* \otimes E  \\
}
}
 \noindent  is a partial flat connection on $E$.  
 If the Poisson  connection  $\nabla_0:E \to TW\otimes E $  does not factor through $T\F_{\sigma}$, then it induces 
a non-trivial  section 
 $$\phi\in H^0(W, N\F_{\sigma}\otimes End(E) )$$
 such that $\phi\wedge \phi =0$.  In fact, 
we have the commutative diagram 
 
  \centerline{
\xymatrix{
E\ar[rdd]_{\phi} \ar[r]^{\nabla} &   TW \otimes E\ar[dd]^{\pi\otimes Id}   \\
  &     \\
    &     N\F_{\sigma} \otimes E.
}
}
 \noindent  where $\phi= (\pi\otimes Id) \circ \nabla$  and $\pi: TW \to  N\F_{\sigma}$ denotes the projection.  
 Now, let $W^0:=W-Sing(\F_{\sigma})\cup Sing(W)$  and consider   $\phi_{0}$  the restriction of $\phi$ on $W^0$.  In order to show that    $\phi\wedge \phi =0$  we follow the argument  in \cite{Wang} and next we extend $\phi_{0}$. Indeed, 
let  $\theta$ be a local  matrix representing the  connection  $\nabla_0$. By flatness we have that
$$
\delta(\theta) +\theta\wedge \theta=0.
$$
Then $\pi(\delta(\theta) )=0$, since $\delta(\theta)$ is tangent to $\F_{\sigma}|_{W^0}$ and $0=\pi(\theta\wedge \theta)= \pi(\theta) \wedge \pi(\theta) $.  Since 
$\pi(\theta)$ is the local matrix of $\phi_0$, we conclude that $\phi_0\wedge \phi_0 =0$, i.e,   $\phi\wedge \phi =0$,  since $Sing(\F_{\sigma})\cup Sing(W)$ has codimension $\geq 2$.

If  $  \sigma_2=0$, then $ (\pi_2)_*[\tilde{\nabla}]\in H^0(Y,  T_Y \otimes End(E_0))$ is a co-Higgs field. 

 If $ (Y,\sigma_2)$  is  generically symplectic, then the Poisson module $(E_0,\nabla_0)$ corresponds to a   meromorphic    flat   connection.  
  
     From now on we  suppose that $  \sigma_2\neq 0$ and    $ (Y,\sigma_2)$  is not  generically symplectic.   
     
   If $\F_{\sigma}$  is  algebraic,  i.e., the  symplectic foliation  $f^*\F_{\sigma}$  is given by the fibration   $ \pi_2: Z=W\times Y \to  Y $.  
 Then,  the connection $ \tilde{\nabla}$ corresponds either 
   to a relative flat connection $\tilde{\nabla}: f^*E \to \Omega^1_{Z|Y}\otimes  f^*E $,  since   we have the isomorphism $\Omega^1_{Z|Y} \simeq T_{Z|Y}$, or the  induced  co-Higgs field $[\tilde{\nabla}]\in H^0(Z,   \pi_2^*T_Y \otimes End( f^*E ))$ is such that 
$  (\pi_2)_*[\tilde{\nabla}]\in H^0(Y,  T_Y \otimes End(E_0 ))$.

If  the symplectic foliation  $\F_{\tilde{\sigma}}$ is not algebraic, then $\sK$ corresponds to the symplectic foliation of the Poisson structure $\sigma_2$, since $\pi_2^{-1} \F_{\sigma_2}=  \F_{\tilde{\sigma}} =f^{-1}\F_{\sigma}=\pi_2^{-1}\sK$. 
 Therefore, as we have seen above,   it is either:
 \begin{itemize}
\item[(i)]  $(E_0,\nabla_0)$ corresponds to a  flat holomorphic  sheaf along the symplectic foliation, or
\\ 
\item[(ii)]  $\nabla_0$  induces 
a non-trivial  section 
 $$\phi_0\in H^0(Y, N\sK\otimes End(E_0) )$$
 such that $\phi_0\wedge \phi_0 =0$. 
\end{itemize}
Consider an embedding  $Y\subset \mathbb{P}^N$ and  let $k:= \dim(\sK) $. 
We can  take a generic
projection  $ q:\mathbb{P}^N  \dashrightarrow  \mathbb{P}^n $, such  that restricted to $Y$ we produce  a finite surjective morphism $ q_{|Y}: Y  \to \mathbb{P}^n $.  We also can   take   a generic rational linear projection $p:\mathbb{P}^n  \dashrightarrow  \mathbb{P}^k $ in such way that    the fibration  $ p\circ (q_{|Y}): Y     \dashrightarrow  \mathbb{P}^k $ is generically transversal to $\sK$. After we  take a Stein factorization   we obtain a rational fibration $ \zeta: Y     \dashrightarrow  B$, with connected fibers, which induces 
 on $Y$ an   algebraic  foliation $\sG$ of dimension equal to $n-k$. 
Consider   the tangency loci   between $\sK$ and $\sG$  given by  
$$S:=\{y\in Y; \omega(\xi)(y)=0\},$$  
where $\omega \in H^0(Y,\Omega_Y^{[n-k]}\otimes \det(N\sK))$ and $\xi\in H^0(Y,\wedge^{k}TY\otimes \det(T_{\sG})^*)$ are the tensors inducing $\sK$  and $\sG$, 
respectively.  Observe that $S$ contain the singular sets of  $\sK$ and $\sG$.  
By transversality the induced map 

\centerline{
\xymatrix{
T_{\sG}  \ar[rr]^{i} \ar@/^2pc/[rrrr]^{\beta} && TY \ar[rr]^{\pi} && N\sK   \\
}
} 
 \noindent   is an isomorphism on $Y^0:=Y-S$.   This gives  us a co-Higgs field 
  $$  \tilde{\phi_0} \in H^0(Y^0,  (T_{\sG} \otimes End(E_0))|_{Y^0} )$$ given by  $\tilde{\phi_0}=\phi_0\circ \beta^{-1}$. 
Denote by $j:Y^0\to Y$ the inclusion map and consider  the sheaf of meromorphic sections with poles of arbitrary order on $S$ given by  
  $$
  T_{\sG} \otimes End(E_0) (*S)=j_* ((T_{\sG} \otimes End(E_0))|_{Y^0}).  
  $$
Thus, we obtain  a meromorphic co-Higgs  field 
  $ \phi \in H^0(Y,  T_{\sG} \otimes End(E_0)\otimes \O_Y(*S))$ which is a meromorphic extension of $\tilde{\phi_0}$.
Recall that $\sK$ and $\sG$ are regular on $Y^0$. Thus,   
it  follows from   \cite[ Corollary 3.3]{Wang}  that the Partial connection on $ N\sK|_{Y^0} \simeq   T_{\sG} |_{Y^0} $ and the Poisson connection $\nabla_0$ induces  a Poisson connection $\tilde{D_0}$ on 
$$(  N\sK^*\otimes End(E_0))|_{Y^0}\simeq (\Omega^1_{Y|B}\otimes End(E_0))|_{Y^0}$$ 
such that 
 $ \tilde{D_0}(   \tilde{\phi_0})=0$. Denoting by $D_0$ its meromorphic extension we have that $D_0( \phi)=0$.

\end{proof}

\begin{remark} \label{conn-Higgs}
 If   $\phi$  denotes a  local representation of  $ \phi_0 $, then from 
 \cite[Proposition 3.2, equation 3.2 ]{Wang}   we conclude  that  $D_0( \phi)=0$ implies that 
$$ \delta(\phi )=0\ \ \ \phi  \wedge \phi  =0 .$$

\end{remark}

Fixed a klt Poisson structure $(X,\sigma)$, consider 
  the  associated  category   of  Poisson modules $  {\text{\rm Rep}}(X,\sigma) $. Also, 
 denote by $ {\text{\rm  Co-Higgs}} (X )$ the category of co-Higgs bundles and 
   $ \text{\rm  Conn}(X,D)$  the category of meromorphic connections along $D$.
 It follows from the proof of Theorem \ref{teo-Morita}    that  via the Morita equivalence  between
  $(X,\sigma)$     and  $   (Y,\sigma_2) $  we obtain the following induced  functors:
\begin{itemize}
\item  $ \pi_*f^*  : {\text{\rm Rep}}(X,\sigma)\to  {\text{\rm  Co-Higgs}} (Y ) $, if $   \pi_*  f^* (\sigma)=0$. 

\item    $\pi_*   f^* : {\text{\rm Rep}}(X,\sigma)\to   {\text{\rm  Conn}}(Y,D ) $, if $ \pi_*  f^* (\sigma)$ is generically symplectic and $D$   is the degeneracy Poisson Divisor.

\end{itemize}

\section{Rank two  $\mathfrak{sl}_2$-Poisson modules }\label{section-sl2}

Let $(X,\sigma)$ be a Poisson projective  variety with Poisson bivector $\sigma$ and denote by $\delta$ the Lie algebroid derivation induced by $\sigma$.
As we  have seen in the Theorem \ref{teo-mori}, the presence of singularities of the  Poisson structure   forces  in certain situations that the Poisson connections     have 
poles along a divisor. In this section we will study the geometry of rank two meromorphic Poisson modules which are trace free.

\begin{defi}
A   rank two holomorphic  vector bundle $E$ on a  projective  Poisson variety $(X,\sigma)$ is called a  {\it meromorphic Poisson module}  if there exists a connection $$\nabla: E \to E\otimes T_X(D)$$ with effective polar divisor $D$, 
such that $\nabla^2=0$. 
If  $\mathrm{tr}(\nabla)=0$  
we say that $(E,\nabla)$ is a {\it meromorphic  $\mathfrak{sl}_2$-Poisson module}.  If $D=\emptyset$ , then $(E,\nabla)$ is a holomorphic Poisson module. 

\end{defi}
Given two  different  $\mathfrak{sl}_2$-Poisson structures   $(E,\nabla_1)$ and $(E,\nabla_2)$, with effective polar divisor $D$, then  $(E,\nabla_1-\nabla_2)$ is a meromorphic   $\mbox{SL}(2,\mathbb{C})$-co-Higgs bundle.

First, we prove the following  Polishchuk's result  \cite{P}  in our context. 

\begin{prop}\label{Prop-Poli}
Let $(E,\nabla)$ be a rank two meromorphic   $\mathfrak{sl}_2$-Poisson module on a  projective  Poisson manifold $X$. Then, there exists a triple of rational vector fields $(v_0,v_1,v_2)$  on $X$ such that 
\begin{equation}\label{MC} 
 \begin{matrix}
\delta(v_0)=\hfill v_0\wedge v_1\\
\delta(v_1)=2v_0\wedge v_2\\
\delta(v_2)=\hfill v_1\wedge v_2
\end{matrix} 
\end{equation}
where $\delta(v)=[v,\sigma]$, $\sigma$ is the Poisson bivector of $X$ and $[\ , \ ]$ denotes the Schouten bracket.  Moreover, $\mathbb{P}(E,\nabla)$ has a   meromorphic  Poisson structure  induced by $\nabla$. 

\end{prop}
\begin{proof}
Since $X$ is projective  we have that $\mathbb{P}(E)$  birationally equivalent  to $X\times \mathbb{P}^1$. In the vector bundle  $X\times \mathbb{C}^2$  we have a    trace free   Poisson connection   given by
$$
 \nabla(Z)=\delta(Z)+\begin{pmatrix}v_1 &v_2
 \\v_0&-v_1\end{pmatrix}\cdot Z,
$$
where $Z=(z_1,z_2)\in \mathbb{C}^2$. The flatness condition  is equivalent to 
\begin{equation} 
 \begin{matrix}
\delta(v_0)=\hfill v_0\wedge v_1\\
\delta(v_1)=2v_0\wedge v_2\\
\delta(v_2)=\hfill v_1\wedge v_2
\end{matrix} .
\end{equation}

Consider the  meromorphic  bivector 
$$
\Sigma_ \sigma = \sigma + (v_0+2v_1z+v_2z^2)\wedge \frac{\partial}{\partial z},
$$ 
where $[1:z]\in  \mathbb{P}^1$ denotes the affine coordinate.  Then $\Sigma$ is Poisson if and only if 
$$
 \begin{matrix}
[v_0,\sigma]=\delta(v_0)=\hfill v_0\wedge v_1\\
[v_1,\sigma]=\delta(v_1)=2v_0\wedge v_2\\
[v_2,\sigma]= \delta(v_2)=\hfill v_1\wedge v_2.
\end{matrix}
$$

\end{proof}

We  observe  that  if  the Poisson connection $(E,\nabla)$  is  such that  $$\nabla:E \to T\F(D)\otimes E\subset  TX  \otimes E(D),$$ then the associated   triple $(v_0,v_1,v_2)$  
satisfying (\ref{MC}) induces on each leaf of $\F$ a 
 transversely projective holomorphic foliation   \cite[Chapter II]{S}. Let $p\in X\setminus Sing(\F)$ and $F_p$ the symplectic leaf of $\F$ passing through $p$.  
 Since $ (\sigma^{\#} )^{-1}\circ \delta |_{F_p}=d \circ (\sigma^{\#} )^{-1} |_{F_p}\ $, where $d$ denotes  the de Rham differential, we have that  
 \begin{equation} 
 \begin{matrix}
 d \circ (\sigma^{\#} )^{-1}(v_0)=(\sigma^{\#} )^{-1}\circ \delta |_{F_p} (v_0)=\hfill   (\sigma^{\#} )^{-1}(v_0)\wedge (\sigma^{\#} )^{-1}(  v_1)\\
d \circ (\sigma^{\#} )^{-1}(v_1)=(\sigma^{\#} )^{-1}\circ \delta |_{F_p} (v_1)=2  (\sigma^{\#} )^{-1}(  v_0)\wedge    (\sigma^{\#} )^{-1}(v_2)\\
d \circ (\sigma^{\#} )^{-1}(v_2)=(\sigma^{\#} )^{-1}\circ \delta |_{F_p} (v_2)=\hfill (\sigma^{\#} )^{-1}( v_1)\wedge (\sigma^{\#} )^{-1}( v_2).
\end{matrix} 
\end{equation}
Now, defining $ (\sigma^{\#} )^{-1} |_{F_p} (v_i)=\omega_i$, we obtain the  transversely projective structure on  $F_p$ given by the triple $(\omega_0,\omega_1,\omega_2)$, where $\omega_0$ is the $1$-form inducing the transversely projective holomorphic foliation  on $F_p$. 
As we already have seen, a  natural way to produce such Poisson structure is by considering a meromorphic connection $\tilde{\nabla}:E \to   E\otimes  \Omega^1_X(D) $  and composing with the anchor map
$\sigma^{\#}:\Omega_X^1\to TX$ we get $\nabla=\sigma^{\#}\circ \tilde{\nabla}:E \to     E\otimes T\F(D) $. See \cite[Proposition 1.8.2]{Fer2} for a more general consideration for principal bundles in the real  category. 

 \begin{remark}
 There is a connection between    transversely projective  holomorphic   foliation and  quantization of symplectic foliation.  
 Biswas  in \cite{Biswas} showed that  for any regular  transversely projective   foliation  $\F$ there is a regular  transversely symplectic foliation $\tilde{\F}$ on its conormal bundle  $N\F^*$. Moreover,  he proves that the restriction of   $\tilde{\F}$ to the complement of the zero section admits a canonical quantization. 
\end{remark}

 \begin{exe}
Consider  a  Poisson structure in $\mathbb{P}^3$ induced in homogeneous coordinates by
 $\sigma=v_0\wedge v_1$,  where $v_0$ and $v_1$ are  degree one polynomial vector fields  satisfying 
 $[v_0,v_1]=0$. In this case we have that 
 $$
\delta(v_i)=L_{v_i} ( v_0\wedge v_1)= 0
$$
for $i=1,2$. 
 
\end{exe}

\begin{exe}
Consider the Poisson structure in $\mathbb{P}^3$ induced in homogeneous coordinates by
$\sigma=v_0\wedge v_1$, where
$$
v_0=z_1\frac{\partial}{\partial z_1} +2z_2\frac{\partial}{\partial z_2}+3z_3\frac{\partial}{\partial z_3},
\\
v_1=-4z_0\frac{\partial}{\partial z_1} -4z_1\frac{\partial}{\partial z_2}-4z_2\frac{\partial}{\partial z_3}.
$$
This Poisson structure corresponds to the exceptional foliation appearing in the Cerveau--Lins Neto's classification \cite{CL}. 
A computation gives us 
\begin{equation} \nonumber
 \begin{matrix}
\delta(v_0)= L_{v_0} (v_0\wedge v_1 )=  v_0\wedge v_1\\
\delta(v_1)=L_{v_1} (v_0\wedge v_1 )=0
\end{matrix} ,
\end{equation}
since $[v_1,v_0]=-v_1$. 
We refer the reader to  \cite[Proposition 8.9.2]{Pym} for a more   conceptual  construction. 
\end{exe}

\begin{remark}
Let $(E,\nabla)$ be a   rank two  holomorphic $\mathfrak{sl}_2$-Poisson module  on a  normal projective  Poisson variety $(X, \sigma) $.    Then  
 $\mathbb{P}(E,\nabla)$ is rationally Morita equivalence  to  $(X, \sigma) $, since  $\pi_*\Sigma_ \sigma=\sigma$.  
In particular, if  $(E,\phi)$  is a  rank two  $\mbox{SL}(2,\mathbb{C})$-co-Higgs bundle,     then  
 $(\mathbb{P}(E ),\Sigma_ \phi)$ is rationally Morita equivalence  to  $ X $ with zero Poisson structure, since $\pi_*\Sigma_ \phi=0$. 
  
  \end{remark}

\begin{prop}\label{proposition-co-higgs}
Let $(E,\phi)$ be a rank two meromorphic   $\mbox{SL}(2,\mathbb{C})$-co-Higgs  bundle  on a  normal projective  Poisson variety $X$.    Then  
 after a birational trivialization  of $ (E, \phi)$  the  co-Higgs   field is of the form 
$$
\begin{pmatrix}f_1 &f_2
 \\f_0&-f_1\end{pmatrix}\otimes v_{\phi},
$$
for some rational  vector field $v_{\phi}\in H^0(X, TX \otimes \sL)$  and rational functions $f_0,f_1,f_2 \in K(X)$. 
  Moreover, the symplectic foliation induced on  $\mathbb{P}(E)$ has dimension two and it is the  pull-back of the one-dimensional foliation $\H_{v_0}$ on $X$ induced by  $v_{\phi}$.  In particular,  if  $\H_{v_0}$ has canonical singularities and $\sL$ is not pseudo-effective,  then  symplectic foliation on  $\mathbb{P}(E)$ is a foliation by rational surfaces. 
\end{prop}

\begin{proof}
We have that $$v_0\wedge v_1=v_0\wedge v_2=v_1\wedge v_2=0.$$ 
We  assume without loss of generality that the rational vector field  $v_0$  is not identically zero.  
Then, there exist  rational functions $f,g,h$ such that  $v_1= f v_0$,  $v_2=g v_0$, and  $v_2=hv_1= h fv_0$, so $g= h  f_1$.  
Therefore, we get the  rational co-Higgs fields
$$
 \Phi \otimes  v= \begin{pmatrix}f & h  f 
 \\1&-f \end{pmatrix}\otimes v_0.
$$

Now,  since the induced Poisson  bivector  is given by
$$
\Sigma= (1+  2fz+   h  fz^2) v_0 \wedge \frac{\partial}{\partial z},
$$
we conclude that the  symplectic $\F_\Sigma$ has dimension two and it is the pull-back of the  foliation $\H_{v_0}$ tangent to $v_0 $. 
Now, if  $\H_{v_0}$ has canonical singularities and $\sL$ is not pseudo-effective,  then  it follows  from   \cite{BM, Druel18}  that   $\H_{v_0}$ is a foliation by rational curves.  Hence, the  leaves  of $\F_\Sigma=\pi^*\H_{v_0}$ are rational surfaces. 
\end{proof}
The author showed  in \cite{C} that  if $(E,\phi)$  is a stable and nilpotent  co-Higgs holomorphic bundle on a compact K\"ahler surface,   then 
 the symplectic foliation induced on  $\mathbb{P}(E)$ is algebraic  with  rational leaves. 

\begin{exe}
 Rayan in  \cite{ Rayan} gave a complete description for the locus of the moduli space of stable holomorphic   $\mbox{SL}(2,\mathbb{C})$-co-Higgs bundles whose underlying bundle is Schwarzenberger. Let $Q$ be an    irreducible element of $H^0(\mathbb{P}^2, \mathcal{O}_{\mathbb{P}^2}(2))$ 
 such that the curve $C=\{Q=0\}\subset \mathbb{P}^2$ is  a nonsingular conic. Consider   a degree two covering $f^Q: \mathbb{P}^1\times \mathbb{P}^1 \to  \mathbb{P}^2 $   branched over  $C=\{Q=0\}$.  The   rank two  vector bundle $f^Q_*\O_{\mathbb{P}^1\times \mathbb{P}^1}(0,k)=V^Q_k$, for each  $k\geq0$ is called Schwarzenberger bundle \cite{Schwarzenberger}. 
  Rayan in  \cite{ Rayan} proved that  Schwarzenberger bundles are    $\mbox{SL}(2,\mathbb{C})$-co-Higgs bundles  with  co-Higgs fields of the form 
  $$
  \Phi \otimes v,
  $$
  where $ \Phi\in H^0( \mathbb{P}^2,End_0(V^Q_k)(-1))$ and $v\in H^0( \mathbb{P}^2, T \mathbb{P}^2(-1))$. 
  It follows from Proposition \ref{proposition-co-higgs} that the dimension two  symplectic foliation  on  $\mathbb{P}(V^Q_k)$ has   rational  leaves which are pull-back of the lines of the pencil tangent to $v\in H^0( \mathbb{P}^2, T \mathbb{P}^2(-1))$.
 
\end{exe}

 \begin{exe}
 We consider  in this example  a Poisson interpretation due to Pym \cite[Section 8.7]{Pym} for degree two pull-back  foliations on $ \mathbb{P}^3$. 
Let $E=\O_{\mathbb{P}^2}\oplus \O_{\mathbb{P}^2}(1)$ and a rational vector field $v_0:\O_{\mathbb{P}^2} \to T\mathbb{P}^2(-1)$.
Then the nilpotent co-Higgs field
induces on  $\mathbb{P}(E)$ a  symplectic foliation such that the contraction of $$ \mathbb{P} (\O_{\mathbb{P}^2}(1))\subset \mathbb{P}(E)  \to \mathbb{P}^3$$  correspond a symplectic foliation given by $v\wedge v_0$, where $v_0\in H^0(\mathbb{P}^3, T\mathbb{P}^3(-1))$  is a  rational vector field which is  tangent to a linear projection $\mathbb{P}^3  \dashrightarrow  \mathbb{P}^2$ and $v$ is a global holomorphic  vector field on $\mathbb{P}^2$. 
 
\end{exe}

\begin{defi}
Let $Z\to X$ be a $\mathbb{P}^1$-bundle. A codimension one holomorphic  foliation $\sG $ on $Z$ is called by {\it  Riccati foliation} 
if it is generically transversal to the $\mathbb{P}^1$-bundle $Z\to X$. 
 \end{defi}

In the next result, as a consequence of Theorem \ref{teo-mori},   we will give a geometric description for  rank two holomorphic  $\mathfrak{sl}_2$-Poisson   modules.

\begin{cor}  
Let  $(E,\nabla)$ be a  rank two  holomorphic  $\mathfrak{sl}_2$-Poisson   module  on a  klt Poisson  projective  variety $(X,\sigma)$.  
 Then there exist   projective
varieties $Y$ and $Z$ with klt singularities, a  quasi-\'etale  Poisson cover $f: W\times Y \to  X$ and  at least one of the following holds.  
\begin{itemize}
\item[(a)]    $(\pi_2)_*f^*(E,\nabla)$  is a  $\mathfrak{sl}_2$ partial   holomorphic  sheaf on $Y$, where $\pi_2$ denotes the projection on $Y$. 

\item[(b)] If $W$ and $Y$ are  generically symplectic, then $(\pi_2)_*f^*(E,\nabla)$ is a rank two  locally free sheaf 
   with a meromorphic    flat   connection with poles on the degeneracy Poisson divisor of $Y$. 

 \item[(c)]  
 If $W$ is    symplectic, then   after a birational trivialization  of $ f^*(E,\nabla)$ the   Poisson connection on the trivial bundle is defined  as 
$$
\tilde{ \nabla}=\delta_{W}+\begin{pmatrix}f_1 &f_2
 \\f_0&-f_1\end{pmatrix}\otimes v,
$$
for some rational  vector field $v$ tangent to $(Y,0)$,  rational functions $f_0,f_1,f_2 \in K(Y)$, and 
$ \delta_{W}$ denotes the  Poisson differential on $W$.
\item[(d)]    There exists a rational  map  $\zeta: Y  \dashrightarrow B$,  
over a variety $B$ with $\dim(B)= \dim(\F_{\alpha})$,   such that $(\pi_2)_*f^*(E,\nabla)$ 
corresponds to a meromorphic    $\mathfrak{sl}_2$-Poisson 
 module  $(E_0, \tilde{ \nabla}) $, such that after a birational trivialization  the  Poisson connection  on the trivial bundle is defined  as 
$$
\tilde{ \nabla}=\delta+\begin{pmatrix}f_1 &f_2
 \\f_0&-f_1\end{pmatrix}\otimes v,
$$
for some rational Poisson vector field $v$ and    rational functions $f_0,f_1,f_2$ on $X$ such that $\{f_i,f_j\}= 0$, for all $i,j$. 
\end{itemize}
\end{cor}
\begin{proof}
From the proof of Theorem \ref{teo-mori}  we have that  there exist   projective
varieties $Y$ and $Z$ with klt singularities and a 
quasi-\'etale  Poisson cover $f: (W\times Y, \tilde{\sigma} )\to (X, \sigma)$ such that $f^*(E,\nabla)$ is a  $\mathfrak{sl}_2$-Poisson module  on $(W\times Y, \tilde{\sigma} )$. 
Then, the result follows from Theorem \ref{teo-mori}. In fact, 
if   $(Y,(\pi_2)_*  f^* (\sigma))$  is    generically  symplectic, then   $(\pi_2)_*f^*(E,\nabla)$ is a rank two  locally free sheaf 
   with a meromorphic    flat   connection with poles on the degeneracy Poisson divisor of $Y$.   
   If   $  (\pi_2)_*  f^* (\sigma)=0$,  then $(\pi_2)_*f^*(E,\nabla)$ is a co-Higgs sheaf with co-Higgs field $(\pi_2)_*f^*(\nabla)=\phi$ and    after we take  a birational trivialization  of $ f^*(E,\nabla)$ the   Poisson connection is defined   as 
   $
\tilde{ \nabla}=\delta_{W}+ (\pi_2)^*\phi. 
$
 Hence, we conclude the part (c) from proposition \ref{proposition-co-higgs}. Finally,  for the  item (d), 
 after  we  take a birational trivialization  of $ f^*(E,\nabla)$ 
we    use the remark \ref{conn-Higgs}   and  Theorem \ref{teo-Morita}, part (d),  in order to conclude
 that  
$$ \delta(\phi )=0,\ \ \ \phi  \wedge \phi  =0 $$ 
with  $\phi$ tangent  a  rational  map  $\zeta: Y  \dashrightarrow B$.  
This implies that   $$v_0\wedge v_1=v_0\wedge v_2=v_1\wedge v_2=0.$$ Then,
there exist  rational functions $f_0,f_1,f_2$ such that  $v_1=f_1v_0$,  $v_2=f_2v_0$, and  $v_1=f_0v_2 $. 
We may assume without loss of generality  that  $v_0\neq0$ and $f_2\neq 0$, since the other cases follow similarly.  By taking the Poisson derivation $\delta$ in  
$v_2=f_2v_0$, we get that $$0=\delta(v_2)=\delta(f_2)\wedge v_0$$ which implies that  $v_0=h_2\delta(f_2)$, for some invertible rational function $h_2$. 
We  also can  conclude that $v_0=h_1\delta(f_1)$. Then
$$
h_1\delta(f_1)=h_2\delta(f_2).
$$
Thus, 
$$\{f_1,f_2\}=\delta(f_1)(f_2)= \frac{h_2}{h_1} \cdot \delta(f_2)(f_2)= \frac{h_2}{h_1} \cdot  \{f_2,f_2\}=0.$$
Now, on the one hand, by using that  $v_1=f_0v_2 $, we obtain that  $v_1=h_0\delta(f_0)$. On the other hand, since $v_1=f_1 h_1\delta(f_1)$ and   $v_2=f_2 h_2\delta(f_2)$ we have that
$$
\{f_0,f_1\}=\delta(f_0)(f_1)= \frac{f_1h_1}{h_0} \cdot \delta(f_0)(f_0)=  0.
$$
and also
$\{f_0,f_2\}=0$. 
\end{proof}

As we have seen above,  the geometric  study of  the symplectic foliation  $\F_{ \nabla}$  reduces, up to a 
quasi-\'etale  Poisson cover,   to the   foliation $\F_{ \nabla_0}$  on    $  \mathbb{P}(E_0,  \nabla_0) \to (Y, \sigma_0)$.

\begin{cor}
 Let $\F_{ \nabla_0}$ be the symplectic foliation induced on    $ \pi:  \mathbb{P}(E_0) \to (Y, \sigma_0)$.  
 Then at least one of the following holds.

\begin{itemize}
\item[(a)] 
$\F_{ \nabla_0}$   is a  dimension 2 foliation   which is a pull-back of a foliation by curves on $(Y,0)$.  

\item[(b)]  
 $ \F_{ \nabla_0}$  is a Riccati  foliation of codimension  one on     $ \mathbb{P}(E_0)$, if  $ (Y, \sigma_0)$ is generically symplectic.     

\item[(c)]  
$\F_{ \nabla_0}$  is a  Riccati  foliation of codimension  one on   $ \mathbb{P}(E_0)$ which  is given by  a morphism  $\sA \to  d_{\text{\rm refl}} \pi(  \pi^{*}(T\F_{ \sigma_0}^*))\subset \Omega_{\mathbb{P}(E_0)}^{[1]}$, where $\sA$ is  a line bundle and  $d_{\text{\rm refl}} \pi: \pi^{*}\Omega_Y^{[1]} \to \Omega_{ \mathbb{P}(E_0) }^{[1]} $  is the  pull-back morphism of reflexive forms.

\item[(d)]   There exist a rational Poisson vector field $v$ generically transversal to $\F_{\sigma_0}$  such that  $\F_{ \nabla_0}$     has   dimension  $2k+2$ and it is the  pull-back of the foliation induced by $v$ and $\F_{\sigma_0}$.  In particular, if $\dim(Y)=2k+1$, then  $ \mathbb{P}(E_0)$ is generically symplectic and there exist a rational Poisson map $\zeta:  Y  \dashrightarrow  B$ generically transversal to  $\F_{\sigma_0}$, where $B$ is a generically symplectic variety with  $\dim(B)=2k$ and  the induced map  $ \mathbb{P}(E_0) \dashrightarrow  B$ is  Poisson. 
  
\end{itemize}
\end{cor}
\begin{proof}
If  $ (Y, \sigma_0)$ is generically symplectic, then  $ \nabla_0$ corresponds   to a meromorphic flat connection.  Hence,   $\F_{ \nabla_0}$ is  a  Riccati  foliation of codimension  one. 
If $ \nabla_0$ corresponds to  a partial  flat meromorphic connection $$E_0\to T\F_{ \sigma_0}^*(D)\otimes E_0 \subset \Omega_Y^1 \otimes E_0,  $$
then $\F_{ \nabla_0}$ is  a   Riccati  foliation of codimension  one,  which is given by  the meromorphic 1-form
$$
\alpha=dz+\omega_0+2\omega_1z+\omega_2z^2
$$
where $\omega_i$'s are meromorphic  sections of  $\pi^{*}(T\F_{ \sigma_0}^* )\subset  \pi^{*}\Omega_Y^{[1]}$. Therefore, $\F_{ \nabla_0}$ is induced by  a morphism  $\sA \to  d_{\text{\rm refl}} \pi(  \pi^{*}(T\F_{ \sigma_0}^*))\subset \Omega_{\mathbb{P}(E_0)}^{[1]}$, where $\sA$ is  a line bundle and  
$$
d_{\text{\rm refl}} \pi: \pi^{*}\Omega_Y^{[1]} \to \Omega_{ \mathbb{P}(E_0) }^{[1]} 
$$
is the  pull-back morphism which there exist by  \cite[Theorem  1.4]{GKKP11}.

Now,   recall  from the proof of Proposition \ref{Prop-Poli}  that   the foliation $\F_{ \nabla_0}$  is induced by the  bivector
$$\Sigma_ {\sigma_0} = \sigma_0 + (v_0+2v_1z+v_2z^2)\wedge \frac{\partial}{\partial z}.$$
If $ \nabla_0$ corresponds to a co-Higgs field on $(Y,0)$, then from Proposition \ref{proposition-co-higgs} we have  that   $\F_{ \nabla_0}$   is a  dimension 2 foliation   which is a pull-back of a foliations by curves. This shows the part (c). 
Finally,  for the case (d) we have the bivector 
$$
\Sigma_ {\sigma_0} = \sigma_0 + v(z)\wedge \frac{\partial}{\partial z},
$$
where $v(z)= \ell v$, with $\ell$ being the non-zero rational function  $f_0+2f_1z+f_2z^2$. 
Thus, $\F_{ \nabla_0}$  is induced by 
$$
\Sigma_ {\sigma_0}^{k+1}=\sum_{j=0}^{k+1}  { k+1\choose j}\ell^{j} \sigma_0^{k+1-j} \wedge  \left(v\wedge \frac{\partial}{\partial z}\right)^{j}= (k+1) \ell \sigma_0^{k} \wedge  v\wedge \frac{\partial}{\partial z}\neq 0
$$
since $ \sigma_0^{k+1} = 0$,  $ \sigma_0^{k}\neq 0$ and   $v$ is transversal to $ \sigma_0^{k}$. 
This   show us that 
$\F_{ \nabla_0}$  is induced by the pull-back of the   foliation  generated by $v$  and $ \sigma_0^{k}$. 
Now, 
if  $\dim(Y)=2k+1$, then the symplectic foliation has codimension one and the rational vector fields $v$ induces a foliation whose leaves are  tangent to the  fibers of $\zeta:  Y  \dashrightarrow  B$. The condition $\delta(v)=L_v\sigma_0=0$, says us that there exist a non-trivial  Poisson bivector $\sigma_B$ such that $\zeta_*\sigma_0=\sigma_B$. 
\end{proof}

\end{document}